\documentclass[a4paper]{article}
\usepackage{graphicx} 
\usepackage{amsmath, amsthm, amsfonts, amssymb, amscd, bm}
\usepackage{verbatim}

\usepackage{url}

\newtheorem{thm}{Theorem}[section]
\newtheorem{cor}[thm]{Corollary}

\theoremstyle{remark}
\newtheorem{rem}[thm]{Remark}

\begin{document}
\title{A Ring Isomorphism and corresponding Pseudoinverses}
\author{Samuel N. Cohen\\ University of Adelaide\\samuel.cohen@adelaide.edu.au
\and Robert J. Elliott\\ University of Adelaide and University of Calgary\\relliott@ucalgary.ca
\and Charles E.M. Pearce\\ University of Adelaide\\charles.pearce@adelaide.edu.au}

\date{\today}
\maketitle

\begin{abstract}
This paper studies the set of $n\times n$ matrices for which all row and column sums equal zero. By representing these matrices in a lower dimensional space, it is shown that this set is closed under addition and multiplication, and furthermore is isomorphic to the set of arbitrary $(n-1)\times (n-1)$ matrices. The Moore-Penrose pseudoinverse corresponds with the true inverse, (when it exists), in this lower dimension and an explicit representation of this pseudoinverse in terms of the lower dimensional space is given. This analysis is then extended to non-square matrices with all row or all column sums equal to zero.
\end{abstract}

\section{Introduction}
In various areas of mathematics, one comes across matrices where all row and column sums are constrained to equal zero. Examples arise in coding theory, stochastic processes, statistics and many other areas where linear algebra is used. Particular examples are the rate matrices of doubly stochastic, continuous time, finite-state Markov chains, the quadratic covariation matrices of any finite-state Markov chain (for example \cite{Cohen2008}) and the `Laplacian matrices' of graph theory.  One problem which arises when studying these matrices is that their inverses do not exist: the vector with all elements one is an eigenvector with corresponding eigenvalue zero, so the determinant of these matrices will be zero. This supports reference to alternative definitions of the inverse, the most common being the Moore-Penrose (pseudo-)inverse. 

The Moore-Penrose inverse has interesting properties in this context, which are derived in this paper. The primary result is that the set of $n\times n$ matrices for which all row and column sums equal zero is isomorphic to the set of arbitrary $(n-1)\times (n-1)$ matrices, with a modified product operation. Under this isomorphism, the Moore-Penrose inverse maps to the true inverse when it exists, and to a modified Moore-Penrose type inverse when it does not. Using this, explicit formulae can be given for the Moore-Penrose inverse, and the range of the projection operator $X^+X$ explicitly described, for $X$ a rank $n-1$, $n\times n$ matrix with all row and column sums equal to zero.

\section{Preliminary Ideas}
We denote by $M_n$ the set of all $n\times n$ matrices with components lying in a field $F$, which we shall assume to be of characteristic zero. The canonical example which we shall discuss is the real matrices; however, the analysis can easily be extended to the complex matrices. We denote by $S_n$ the subset of matrices for which all row and column sums are equal to zero. For any matrix $A$, a \textit{pseudoinverse} is any matrix $A^+$ which satisfies
\[AA^+A=A.\]
If furthermore $A^+$ satisfies 
\[\begin{split}
A^+AA^+&=A^+\\
A^+A&=(A^+A)^*\\
AA^+&=(AA^+)^*,\end{split}\]
where $[.]^*$ denotes the conjugate transpose, it is said to be the \textit{Moore-Penrose (pseudo)inverse} of $A$. It is well known that the Moore-Penrose inverse of any matrix $A$ exists and is unique (see, for example, \cite{Wong1970}). If $A$ is invertible then $A^+=A^{-1}$ is the only pseudoinverse, and if $A$ has a left- or a right-inverse then this will be the Moore-Penrose inverse. A good general treatment of generalised inverses can be found in \cite{Campbell1980}.

\section{Representation}

\begin{thm} \label{thm:biject}
There is a bijection $\phi$ between $S_{n+1}$ and $M_{n}$, given by the map
\[ \phi: M_n\rightarrow S_{n+1}, X \mapsto J_n^* X J_n\]
where \[J_n:=[I_n|-\mathbf{1}],\]
an identity matrix with a column of $-1$'s added to the right hand side.
\end{thm}
\begin{proof}
Consider an element $\tilde X$ of $S_{n+1}$. As all the row sums are zero, that is, for any $i$, $\sum_je_j^*\tilde Xe_i=0$, we know that \[e_{n+1}^*\tilde Xe_i = -\sum_{j=1}^n e_j^*\tilde Xe_i.\]  Similarly \[e_{i}^*\tilde Xe_{n+1} = -\sum_{j=1}^n e_i^*\tilde Xe_j\] for all $i$.
Therefore, if $X\in M_n$ is the submatrix of $\tilde X$ excluding the $(n+1)$th row and column, we know that 
\[e_{i}^*\tilde Xe_{n+1} = -\sum_{j=1}^n e_i^*\tilde Xe_j=-e_i^*X\mathbf{1},\]
and similarly for the rows of $\tilde X$.

Therefore, for any $\tilde X\in S_n$, we have the relationship 
\[\tilde X = \left[\begin{array}{c|c} X & -X\mathbf{1}\\\hline-\mathbf{1}^*X & \mathbf{1}X\mathbf{1}\end{array}\right]=J_n^* X J_n=\phi(X).\]
 This establishes the surjectivity of $\phi$. Injectivity follows from the fact $J_n$ is right-invertible.
\end{proof}
\begin{cor} If $\tilde X = J_n^* X J_n$, then $\tilde X$ and $X$ have the same rank.\end{cor}
\begin{proof}
$J_n$ is right invertible, and so no degrees of freedom are lost under this transformation. 
\end{proof}
\begin{cor} $\tilde X = \tilde X^*$ if and only if $X = X^*$.\end{cor}
\begin{proof}It is clear that
\[\tilde X ^* = (J_n^* X J_n)^* = J_n^* X^* J_n,\]
the result follows as $J_n$ is right invertible.
\end{proof}

\begin{rem}It is clear that the choice of placing the column of $-1$'s on the right hand side of $J_n$ is arbitrary, and that this column can equivalently be placed elsewhere. This will not change the value of $J_n J_n^*$, and so the following analysis will be unchanged. Similarly a set of matrices with other linear dependencies could be considered, with the following holding \textit{mutatis mutandis}.
\end{rem}

\begin{thm} \label{thm:Isom}
If $\times$ denotes standard matrix multiplication and $\circ$ denotes the `twisted' product $X\circ Y=XK_nY$, $K_n:= J_nJ_n^*$, where $X$ and $Y$ are in $M_n$, then  $\phi$ is an isomorphism between the rings $(S_{n+1}, +, \times)$ and $(M_n, +, \circ)$. 
\end{thm}
\begin{proof}
Suppose
\[\begin{split}
\tilde X &= J_n^* X J_n,\\
\tilde Y &= J_n^* Y J_n.\end{split}\]
 We have
\[\tilde X + \tilde Y = J_n^* [X+Y] J_n\] 
and 
\[\tilde X \tilde Y = J_n^* X J_n J_n^* Y J_n = J_n^* X K_n Y J_n,\]
where $K_n := J_nJ_n^*$. It is then clear that $\tilde X \tilde Y$ has a representation $J_n^*(XK_n Y)J_n$. By Theorem \ref{thm:biject}, this implies that $S_n$ is closed under addition and matrix multiplication and hence forms an algebraic ring.
\end{proof}

We now extend our definition of the Moore-Penrose inverse to the ring $(M_n, +, \circ)$ by replacing each instance of matrix multiplication by the $\circ$ operation. For notational clarity, we shall denote a classical inverse $X^{-1}$, a classical Moore-Penrose inverse $X^{+}$, an inverse under $\circ$ by $X^\ominus$ and a Moore-Penrose inverse under $\circ$ by $X^\oplus$. 

\begin{thm} \label{thm:MPIsom}
$\phi$ maps all pseudoinverses under $\times$ to pseudoinverses under $\circ$. In particular it preserves the Moore-Penrose inverse.
\end{thm}
\begin{proof}
We simply check that $\tilde X^+ = J_n^* X^\oplus J_n$ satisfies the required properties for a (Moore-Penrose) pseudoinverse whenever $X^{\oplus}$ satisfies the requirements under $\circ$, and vice versa.
\end{proof}

Note that $K_n$ has a simple form, namely 
\[K_n =J_nJ_n^*= \left[\begin{array}{cccc}2 &1&\cdots &1\\ 1&2&\cdots&1\\ \vdots &\vdots&\ddots&\vdots\\ 1 & 1 &\cdots & 2\end{array}\right]\in M_n\]
with inverse
\[K_n^{-1} = \frac{1}{n+1}\left[\begin{array}{cccc}n &-1&\cdots &-1\\ -1&n&\cdots&-1\\ \vdots &\vdots&\ddots&\vdots\\ -1 & -1 &\cdots & n\end{array}\right].\]

\begin{thm}
The ring $(S_{n+1}, +, \times)$ has a unique multiplicative identity \[\phi(K_n^{-1})=J_n^* K_n^{-1} J_n = I_{n+1} - \frac{1}{n+1} \mathbf{1}_{(n+1)\times(n+1)},\] where $\mathbf{1}_{(n+1)\times(n+1)}$ here denotes the matrix with all entries $1$. (This differs from the usual matrix identity $I_{n+1}\notin S_{n+1}$.)
\end{thm}
\begin{proof}
Under the operation $\circ$, $K_n^{-1}$ is clearly an identity element and is unique (as $K_n$ has a unique inverse). Hence it is the unique identity in $(M_n, +, \circ)$ and therefore in $(S_{n+1}, +, \times)$ by isomorphism.

This fact could also have been observed directly, as the vector with all entries one is the only vector in the left and right null spaces of every matrix in $S_{n+1}$.
\end{proof}

\section{Pseudoinverses}
We now move to the key result of this analysis -- a representation of the Moore-Penrose pseudoinverses of matrices in $S_{n+1}$. 

\begin{thm}
For any rank $n$ matrix $\tilde X$ in $S_{n+1}$, the Moore-Penrose inverse of $\tilde X$ is the unique pseudoinverse of $\tilde X$ in $S_{n+1}$, and is given by 
\[\tilde X^+=J_n^*K_n^{-1}X^{-1} K_n^{-1}J_n,\]
where $\tilde X = J_n^*X J_n$.
\end{thm}
\begin{proof}
In this case, $\tilde X\in S_{n+1}$ is of rank $n$ and has representation $J_n^* X J_n$. It is clear that $X$ is of full rank, and so has a unique inverse in $(M_n, +, \circ)$, that is, a matrix $X^\ominus$ such that \[X\circ X^\ominus = X^\ominus\circ X = K_n^{-1}.\] Inspection shows that this corresponds to the regular inverse of $X$ through the formula \[X^\ominus =K_n^{-1}X^{-1} K_n^{-1}.\] 

As $(M_n, +, \circ)$ and $(S_{n+1}, +, \times)$ are isomorphic (by Theorem \ref{thm:Isom}), and $X^\ominus$ is the (Moore-Penrose) inverse of $X$ in $M_n$, this leads to the Moore-Penrose inverse of $\tilde X$ having a representation \[\tilde X ^+=J_n^*X^\ominus J_n = J_n^*K_n^{-1}X^{-1} K_n^{-1}J_n\] by Theorem \ref{thm:MPIsom}. 
\end{proof}

\begin{cor}
For any rank $n$ matrix $\tilde X$ in $S_{n+1}$, we have
\[\tilde X\tilde X^+ = \tilde X^+\tilde X = J_n^* K_n^{-1} J_n = I_{n+1} - \frac{1}{n+1} \mathbf{1}_{(n+1)\times(n+1)}\]
and hence $\tilde X^+$ acts as a true inverse in $S_{n+1}$.
\end{cor}

Now suppose that $\tilde X \in S_{n+1}$ is of rank $<n$, and has representation $J_n^* X J_n$. Clearly $X$ will have multiple pseudoinverses in $(M_n,+,\circ)$, as it is of incomplete rank. To find the Moore-Penrose inverse of $\tilde X$, we wish to find the Moore-Penrose inverse $X^\oplus$ under $\circ$. Expanding the conditions for a pseudoinverse, along with our earlier result for rank $n$ matrices, suggests an inverse of the form $K_n^{-1} X^+ K_n^{-1}$. Unfortunately this does not satisfy the symmetry requirements under $\circ$, (and will give a different pseudoinverse). 

Provided we are dealing with real or complex matrices, we know that $K_n$ is symmetric and positive definite and therefore has a Cholesky decomposition $K_n =k_nk_n^*$, where $k_n$ is an invertible, lower triangular $n\times n$ matrix (in fact, we can use any square matrix $k_n$ with $K_n=k_nk_n^*$). 

\begin{thm}
For any real or complex matrix $\tilde X=J_n^*XJ_n\in S_{n+1}$, the Moore-Penrose inverse of $\tilde X$ lies in $S_{n+1}$, is equivalent to the Moore-Penrose inverse of $X$ in $(M_n, +, \circ)$ and is given by the equations
\[\tilde X^+ = J_n^*(k_n^{-1})^*(k_n^* X k_n)^+ k_n^{-1}J_n\]
and
\[\tilde X^+ = J_n^*\lim_{\delta\downarrow 0}(X^* K_n XK_n + \delta I)^{-1}X^*J_n.\]
\end{thm}
\begin{proof}
Expanding $K_n^{-1}$ and $K_n$ using the decomposition $K_n = k_nk_n^*$, we find that the matrix
\[(k_n^{-1})^*(k_n^* X k_n)^+ k_n^{-1}\]
will satisfy the requirements for the Moore-Penrose inverse. This is done in full in the more general non-square case below.

As a Moore-Penrose inverse can be expressed as the limit of a perturbed solution (see, for example, \cite{Nair1997}), we also have the following useful formula
\[\begin{split}
(k_n^* X k_n)^+ &= \lim_{\delta\downarrow 0}(k_n^* X^* K_n X k_n+ \delta I)^{-1}k_n^* X^* k_n\\
&= \lim_{\delta\downarrow 0}(k_n^* X^* K_n X k_n+ \delta k_n^*(k_n^{-1})^* k_n^{-1}k_n)^{-1}k_n^* X^* k_n\\
&= \lim_{\delta\downarrow 0}k_n^{-1}(X^* K_n X+ \delta K_n^{-1})^{-1}X^* k_n\\
\end{split}\]
and so
\[\begin{split}
(k_n^{-1})^*(k_n^* X k_n)^+ k_n^{-1} &= \lim_{\delta\downarrow 0}K_n^{-1}(X^* K_n X + \delta K_n^{-1})^{-1}X^* \\
&= \lim_{\delta\downarrow 0}(X^* K_n XK_n + \delta I)^{-1}X^*.\end{split}\]
\end{proof}

\section{Non-square matrices}
This methodology can be applied to give explicit results for Moore-Penrose inverses for other types of matrices. We now consider the case for real or complex $\tilde X$, where $\tilde X$ has all row and column sums equal to zero but we do not assume that $\tilde X$ is square.

Using the same logic as above, we can write $\tilde X$ as 
\[\tilde X = J_m^*XJ_n\]
for some $m$ and $n$, where $X$ will now be an $m\times n$ matrix.  We again define $K_m = J_mJ_m^* = k_mk_m^*$ for $k_m$ upper triangular and invertible.

\begin{thm}
 For $\tilde X = J_m^* X J_n$, $\tilde X^+$ is given by
\[\tilde X^+ = J_n^*(k_n^{-1})^*(k_m^* X k_n)^+ k_m^{-1}J_m.\]
\end{thm}
\begin{proof}
Inspection verifies that all these matrices are conformable so we need only verify that this formula satisfies the assumptions for the Moore-Penrose inverse.
\[\begin{split}
\tilde X \tilde X^+ \tilde X&=J_m^*XJ_n J_n^*(k_n^{-1})^*(k_m^* X k_n)^+ k_m^{-1}J_m J_m^*XJ_n\\
&=J_m^*Xk_n k_n^*(k_n^{-1})^*(k_m^* X k_n)^+ k_m^{-1}k_m k_m^*XJ_n\\
&=J_m^*(k_m^{-1})^*(k_m^*Xk_n) (k_m^* X k_n)^+ (k_m^*Xk_n)k_n^{-1}J_n\\
&=J_m^*XJ_n\\
\tilde X^+\tilde X\tilde X^+&=J_n^*(k_n^{-1})^*(k_m^* X k_n)^+ k_m^{-1}J_mJ_m^*XJ_nJ_n^*(k_n^{-1})^*(k_m^* X k_n)^+ k_m^{-1}J_m\\
&=J_n^*(k_n^{-1})^*(k_m^* X k_n)^+ k_m^{-1}J_m\\
(\tilde X \tilde X^+)^*&=(J_m^*XJ_n J_n^*(k_n^{-1})^*(k_m^* X k_n)^+ k_m^{-1}J_m)^*\\
&=(J_m^*Xk_n(k_m^* X k_n)^+ k_m^{-1}J_m)^*\\
&=(J_m^*(k_m^{-1})^*(k_m^*Xk_n)(k_m^* X k_n)^+ k_m^{-1}J_m)^*\\
&=J_m^*(k_m^{-1})^*(k_m^*Xk_n)(k_m^* X k_n)^+ k_m^{-1}J_m
\end{split}\]
and the final condition follows as the third. Therefore, the conditions for the Moore-Penrose inverse are satisfied.
\end{proof}

We can also deal with (possibly non-square) matrices where only the row sums are equal to zero (or equivalently, where only the column sums are equal to zero). In this case we have the representation
\[\tilde X = X J_n\]
with $J_n$ as before. This gives the following.
\begin{thm}\label{thm:rowsonly}
Let $\tilde X = X J_n$, with $X$ a left-invertible matrix. Then the Moore-Penrose inverse of $\tilde X$ is given by 
\[\tilde X^+ = J_n^* K_n^{-1} X^+.\]
\end{thm}
\begin{proof}
As $X$ is left-invertible, $X^+X=I$. Therefore
\[\begin{split}
\tilde X (J_n^* K_n^{-1} X^+) &= XJ_n J_n^* K_n^{-1}X^+\\&=XX^+\\
(J_n^* K_n^{-1} X^+)\tilde X &=J_n^* K_n^{-1} X^+XJ_n\\&=J_n^* K_n^{-1}J_n\end{split}\]
both of which are symmetric. It is clear that 
\[\tilde X J_n^* K_n^{-1} X^+ \tilde X =X J_n J_n^* K_n^{-1} X^+ X J_n= XJ_n =\tilde X\] 
and 
\[J_n^* K_n^{-1} X^+ \tilde X J_n^* K_n^{-1} X^+=J_n^* K_n^{-1} X^+ X J_n J_n^* K_n^{-1} X^+ = J_n^* K_n^{-1} X^+.\] Hence $J_n^* K_n^{-1} X^+ $ satisifes all the requirements for the Moore-Penrose inverse of $\tilde X$.
\end{proof}

\begin{thm}
Let $\tilde X = J_n^*X$, with $X$ a right-invertible matrix. Then the Moore-Penrose inverse of $\tilde X$ is given by 
\[\tilde X^+ = X^+K_n^{-1}J_n.\]
\end{thm}
\begin{proof}This follows as in Theorem \ref{thm:rowsonly}.\end{proof}

\section{Extra Zero Rows}
 A particular example of interest is for $\tilde X\in S_{n+1}$, where $\tilde X$ is of rank $m$ and contains $n-m$ rows and columns which are invariably zero. An example of such a matrix with $n=3$, $m=2$ is
\[\left[\begin{array}{cccc} 1 & 0 &1 & -2\\3&0&-1&-2\\0&0&0&0\\-4&0&0&3\end{array}\right].\]

This can be written through the linear transformation
\[\left[\begin{array}{ccc} 1 &0 \\0&1\\0&0\\-1&-1\end{array}\right]\left[\begin{array}{cc} 1 &1 \\3&-1\end{array}\right]\left[\begin{array}{cccc} 1 &0 & 0&-1\\0&0&1&-1\end{array}\right]= J_{2,\{3\}}^*\left[\begin{array}{cc} 1 &1 \\3&-1\end{array}\right]J_{2,\{2\}},\]
where $J_{2,\{3\}}$ refers to the $J_2$ matrix as defined above, with an additional column of zeros inserted into the third place.

For any $J_{m, \mathbf{a}}$, where $\mathbf{a}$ lists the columns where zeros are inserted, 
\[J_{m,\mathbf{a}} J_{m, \mathbf{a}}^* = K_m.\]

\begin{thm}
For a fixed $\mathbf{a}$ with $m$ entries, $m<n$, let $S_{n+1}^\mathbf{a}$ be the set of $(n+1)\times(n+1)$ matrices with all row and column sums zero, and with zeros in those rows and columns listed in $\mathbf{a}$. Then there exists an isomorphism
\[\phi: (M_{n-m}, +, \circ)\rightarrow (S_{n+1}^\mathbf{a}, +, \times), X\mapsto J_{m,\mathbf{a}}^* X J_{m, \mathbf{a}}\]
where $X\circ Y = X K_m Y$ as before.
\end{thm}
\begin{proof}
The proof of this follows exactly as for Theorem \ref{thm:Isom}.
\end{proof}

\begin{thm}
The Moore-Penrose inverse of $\tilde X= J_{m,\mathbf{b}}^* X J_{m, \mathbf{a}}$, where $\tilde X$ is of rank $m$, is given by
\[\tilde X^+ = (J_{m,\mathbf{b}}^* X J_{m, \mathbf{a}})^+ = J_{m,\mathbf{a}}^* K_m^{-1}X^{-1}K_m^{-1} J_{m, \mathbf{b}}.\]
\end{thm}
\begin{proof}
If
\[\tilde X = J_{m,\mathbf{b}}^* X J_{m, \mathbf{a}}\]
then 
\[\tilde X^* = J_{m,\mathbf{a}}^* X^* J_{m, \mathbf{b}}.\]
We can see that if $\tilde X$ is of rank $m$, so is its equivalent $m\times m$ matrix $X$. Therefore, this has a true inverse. Using the fact that $J_{m,\mathbf{a}} J_{m, \mathbf{a}}^*=J_{m,\mathbf{b}} J_{m, \mathbf{b}}^* = K_m$, we can then expand the definition of the Moore-Penrose inverse to give the desired result.
\end{proof}

\begin{rem}
Note that this theorem does not require that $\tilde X$ is square, as $\mathbf{a}$ and $\mathbf{b}$ can list different numbers of positions where zeros are inserted.
\end{rem}

We can, therefore, establish the following general property for all matrices of this form.
\begin{thm}
For $\tilde X= J_{m,\mathbf{b}}^* X J_{m, \mathbf{a}}$ of rank $m$,
\[\begin{split}
\tilde X^+ \tilde X &= J_{m,\mathbf{a}}^*K_m^{-1}X^{-1}K_m^{-1} J_{m, \mathbf{b}}J_{m,\mathbf{b}}^* X J_{n, \mathbf{a}}\\
&= J_{n,\mathbf{a}}^* K_m^{-1} J_{m, \mathbf{a}}.\end{split}
\]
\end{thm}

\section{A final fact}
A final useful result is that, for matrices of this form, if $\tilde X=J_{m,\mathbf{b}}^* X J_{m, \mathbf{a}}$ is of rank $m$, then
\[\begin{split}
\tilde X^+ \tilde X &= J_{m,\mathbf{a}}^* K_m^{-1} J_{m, \mathbf{a}}\\
&= J_{m,\mathbf{a}}^*\left[I_{m}-\frac{1}{m+1}\mathbf{1}_{m\times m}\right]J_{m,\mathbf{a}}.\end{split}\]
where $\mathbf{1}_{m\times m}$ again refers to a matrix with all entries one.

\begin{thm}
Let $M$ be any matrix with $m$ rows, and $\tilde X$ any rank $m$ square matrix of the form $\tilde X = J_{m,\mathbf{b}}^* X J_{m, \mathbf{a}}$. If $\tilde M = J_{m,\mathbf{a}}^*M$, then \[\tilde X^+\tilde X \tilde M = \tilde M.\]
\end{thm} 
\begin{proof}
We have
\[\tilde X^+\tilde X \tilde M = J_{m,\mathbf{a}}^* K_m^{-1} J_{m, \mathbf{a}}J_{m,\mathbf{a}}^*M = J_{m,\mathbf{a}}^* K_m^{-1} K_m M = \tilde M.\]
\end{proof}
Considering this form of $\tilde M$, it is clear that these are the matrices with column sums of zero and with zeros in the rows denoted by $\mathbf{a}$, that is, corresponding to the columns of zeros in $\tilde X$. 

\begin{thm}
The vectors of the form $\tilde M = J_{m, \mathbf{a}} M$ form an $m$-dimensional subspace spanning the range of the projection operator $\tilde X^+ \tilde X$.
\end{thm}

\section{Conclusion}
We have examined some of the properties of matrices with all row and column sums equal to zero, particularly in relation to their Moore-Penrose inverses. The set of these matrices is isomorphic to the set of matrices of smaller dimension, which easily shows that they form a ring with identity and gives general results for the calculation of the pseudoinverses within this ring. We have shown that the true inverse in this smaller ring, if it exists, corresponds in some sense to the Moore-Penrose inverse for the original matrices. If we consider the subset of matrices with the extra restriction that a certain number of rows and columns are zero, then a further explicit formula for the projection operator $X^+X$ can be obtained, which characterises its range. These results have applications in the theory of Markov chains and other areas.

\bibliographystyle{plain}
\bibliography{General}
\end{document}